\title{Normal approximation and almost sure central limit theorem for non-symmetric Rademacher functionals}
 \author{\small Guangqu ZHENG\footnote{\href{mailto:guangqu.zheng@uni.lu}{guangqu.zheng@uni.lu}, URM, Universit\'e du Luxembourg, Grand-Duch\'e de Luxembourg}}
\date{}

\documentclass[10pt]{article}
\usepackage{geometry}
\usepackage[T1]{fontenc}
 \usepackage[titletoc,toc,title]{appendix}

\usepackage{color}
\usepackage{bigints}
\usepackage{fancyhdr}
\usepackage[pdftex]{graphicx}

\usepackage{hyperref}

\usepackage{amssymb}

\usepackage{mathrsfs,mathtools}
\usepackage[english]{babel}

\usepackage{amsmath,amsthm, bbm}

\usepackage{dsfont}
\usepackage{varioref}

\newcommand*\Laplace{\mathop{}\!\mathbin\bigtriangleup}

\numberwithin{equation}{section}

\theoremstyle{definition}

\newtheorem{rem}{Remark}[section]
\newtheorem{lem}{Lemma}[section]
\newtheorem{defn}{Definition}[section]
\newtheorem{thm}{Theorem}[section]
\newtheorem{prop}{Proposition}[section]
\newtheorem{cor}{Corollary}[section]
\labelformat{thm}{Theorem~#1}
\labelformat{lem}{Lemma~#1}

\newtheorem{example}{Example}[section]

\def\R{\mathbb{R}}
\def\P{\mathbb{P}}

\def\wt{\widetilde}

\def\N{\mathbb{N}}

\def\H{\mathfrak{H}}
\def\F{\mathscr{F} }

\def\DD{\mathbb{D} }

\def\CC{\mathscr{C} }

\def\E{\mathbb{E}}
\def\G{\mathscr{G} }
\def\to{\rightarrow}

\def\lto{\longrightarrow}
\def\lmto{\longmapsto}

\def\bv{\big\vert}
\def\Bv{\Big\vert}
\def\1{\mathbbm{1}}

\def\H{\mathfrak{H}}

\def\Var{\textbf{Var}}

\begin{document}

\maketitle

\begin{abstract}
In this work, we study the normal approximation and almost sure central limit theorems for some functionals of an independent sequence of Rademacher random variables. In particular, we provide a new chain rule that improves the one derived by Nourdin, Peccati and Reinert(2010) and then we deduce the bound on Wasserstein distance for normal approximation using the (discrete) Malliavin-Stein approach. Besides, we are  able to give the almost sure central limit theorem for a sequence of random variables inside a fixed Rademacher chaos using the \emph{Ibragimov-Lifshits criterion}. 
\end{abstract}

\section{Introduction}

This work is  devoted to the study of discrete Malliavin-Stein approach for two kinds of Rademacher functionals: 
\begin{itemize}
\item[(S)] $Y_k, k\in\N$ is a sequence of independent identically distributed ({\it i.i.d}) Rademacher random variables, \emph{i.e.} $\P(Y_1 = -1) = \P(Y_1 = 1) = 1/2$.  $F = f(Y_1, Y_2, \cdot\cdot\cdot )$, for some nice function $f$, is called a (symmetric) Rademacher functional over $(Y_k)$.

\item[(NS)]  $X_k, k\in\N$ is a sequence of independent non-symmetric, non-homogeneous Rademacher random variables, that is, $\P(X_k = 1) = p_k$, $\P(X_k = -1) = q_k$ for each $k\in\N$. Here $1 - q_k = p_k\in(0,1)$ for each $k\in\N$.   Of course this sequence reduces to the i.i.d. one when $p_k = q_k = 1/2$ for each $k$.  $G = f(X_1, X_2, \cdot\cdot\cdot )$, for some nice function $f$, is called a (non-symmetric) Rademacher functional over $(X_k)$. Usually, we  consider the corresponding normalised sequence $(Y_k, k\in\N)$ of $X_k$, that is, $Y_k:= (X_k - p_k + q_k) \cdot (2\sqrt{p_kq_k} )^{-1}$.

\end{itemize}
From now on, we write (S) for the symmetric setting, and (NS) for the non-symmetric, non-homogeneous setting.

Now let us explain several terms in the title.  Malliavin-Stein method stands for the  combination of two powerful tools in probability theory: Paul Malliavin's differential calculus and Charles Stein's method of normal approximation. This  intersection   originates from the seminal paper \cite{NP1} by Nourdin and Peccati, who were able to associate a quantitative bound to the remarkable fourth moment theorem established by Nualart and Peccati \cite{FMT} among many other things. For a comprehensive overview, one can refer to the website \cite{website} and the recent monograph \cite{bluebook}.

  This method has found its extension to discrete settings:   for the Poisson setting, see \emph{e.g.} \cite{PSTU, cengbo_thesis}; for the Rademacher setting, the paper \cite{NPR} by Nourdin, Peccati and Reinert was the first one to carry out the analysis of normal approximation for Rademacher functionals (possibly depending on infinitely many Rademacher variables) in the setting (S), and they were able to get a sufficient condition in terms of contractions for a central limit theorem (CLT)  inside a fixed Rademacher chaos $\CC_m$ (with $m\geq 2$), see Proposition \ref{nishabi} for the precise statement.
  
     In the Rademacher setting, unlike the Gaussian case, one does not have the chain rule like $Df(F) = f'(F) DF$ for  $f\in C^1_b(\R)$ and Malliavin differentiable random variable $F$ (see \cite[Proposition 2.3.7]{bluebook}), while  an approximate chain rule (see \eqref{oldchain}) is derived in    \cite{NPR} and it requires quite much regularity of the function $f$. As a consequence, the authors of \cite{NPR} had to use smooth test functions when they applied the Stein's estimation: roughly speaking, for nice centred Rademacher functional $F$ in the setting (S), for $h\in C_b^2(\R)$, $Z\sim\mathscr{N}(0,1)$, one has (see \cite[Theorem 3.1]{NPR})
     \begin{align}\label{MSbdd_NPR}
    &\qquad  \bv \E\big[ h(F) - h(Z) \big] \bv  \notag \\
     &\leq \min\big( 4\| h\| _\infty, \| h''\|_\infty\big) \cdot  \E\Big[ \bv 1 - \langle DF, -DL^{-1}F \rangle_\H \big\vert \Big] +\frac{20}{3} \| h''\| _\infty \E\Big[ \big\langle \vert DL^{-1}F\vert, \vert DF \vert^3 \big\rangle_\H \Big] \,\, ,
     \end{align}
where the precise meaning of the above notation will be explained in the  Section 2.   

Krokowski, Reichenbachs and Th\"ale, carefully using a representation of the discrete Malliavin derivative $Df(F)$ and the fundamental theorem of calculus instead of the approximate chain rule \eqref{oldchain}, were able to deduce the Berry-Ess\'een bound in  \cite[Theorem 3.1]{KRC14} and its non-symmetric analogue in \cite[Proposition 4.1]{KRC15}: roughly speaking, for nice centred Rademacher functional $F$ in the setting (NS), 
\begin{align}
&\quad d_{K}\big(F, Z \big) : = \sup_{x\in\R} \Bv \P(F\leq x) -  \P(Z\leq x) \Bv \label{Kkkk_bdd} \\
&\leq  \E\Big[ \bv 1 - \langle DF, -DL^{-1}F \rangle_\H \big\vert \Big] + \frac{\sqrt{2\pi}}{8} \E\Big[ \big\langle  \frac{1}{\sqrt{pq}} \vert DL^{-1}F\vert, \vert DF \vert^2 \big\rangle_\H \Big] \label{BSbdd1} \\
&\quad + \frac{1}{2}\E\Big[ \big\langle \vert F\cdot DL^{-1}F\vert,  \frac{1}{\sqrt{pq}} \vert DF \vert^2 \big\rangle_\H \Big] + \sup_{x\in\R } \E\Big[ \big\langle \vert DL^{-1}F\vert,  \frac{1}{\sqrt{pq}}( DF)\cdot\1_{(F>x)} \big\rangle_\H \Big] \,\, . \label{BSbdd2}
\end{align}
The quantity $d_K(F,Z)$ defined in \eqref{Kkkk_bdd}  is called the Kolmogorov distance between $F$ and $Z$.

 For the setting (NS), the corresponding analysis including normal approximation and Poisson approximation has been taken up  in \cite{ET14,Krokowski15,KRC15,PT15}. 

  In this paper, we give a neat chain rule (see \ref{chainrule}), from which we are able to derive a bound on the Wasserstein distance 
  $$d_W(F, Z) : = \sup_{\| f' \| _\infty \leq 1}  \Bv \E\big[ f(F) - f(Z) \big] \Bv  $$ in both settings (NS) and (S), see \ref{NPbdd} and related remarks.

Another contribution of this work is  the almost sure central limit theorem  (ASCLT in the sequel) for Rademacher functionals.  We first give the following

\begin{defn} Given a sequence $\big( G_n, n\in\N \big)$ of real random variables  convergent in law to $Z\sim\mathscr{N}(0,1)$, we say the  ASCLT holds for $(G_n)$, if almost surely, for any bounded continuous $f: \R\to\R$, we have 
\begin{align}\label{asclt_def}   \frac{1}{\log n} \sum_{k=1}^n \frac{1}{k} f(G_k) \lto \E\big[ f(Z) \big] \,\, , \end{align}
as $n\to+\infty$.  In the definition, $\log n$ can be replaced by $\gamma_n : =  \sum_{k=1}^n k^{-1}$, ($\gamma_n - 1 \leq \log n \leq \gamma_n$).  Note  the condition \eqref{asclt_def} is equivalent to that   the random probability measure $\gamma_n^{-1} \sum_{k=1}^n  k^{-1} \cdot \delta_{G_k} $ weakly converges to the  standard Gaussian measure  almost surely, as $n\to+\infty$.

\end{defn}
The following criterion, due to Ibragimov and Lifshits,  gives a sufficient condition for the ASCLT.

\paragraph{Ibragimov-Lifshits criterion} 
\begin{align}\label{ILcondition}
\sup_{\vert t\vert \leq r} \sum_{n\geq 2} \frac{\E\big( \vert \Laplace_n(t) \vert^2 \big)}{n\gamma_n} < +\infty \,\, ,\quad \text{for every $r > 0$,}
\end{align}
where  
  \begin{equation}\label{difference}
  \Laplace_n(t) = \frac{1}{\gamma_n} \sum_{k=1}^n \frac{1}{k} \Big[ e^{i  t G_k } -  e^{-t^2/2} \Big] \, .
  \end{equation}
If $G_k\xrightarrow{\text{law}} Z\sim\mathscr{N}(0,1)$ and  \eqref{ILcondition} is satisfied, then the ASCLT holds for $(G_k)$. See \cite[Theorem 1.1]{ILpaper}.

The ASCLT  was first stated by Paul L\'evy without proof in the 1937 book ``{\it Th\'eorie de l'addition des variables al\'eatoires} '' and rediscovered by Brosamler \cite{Brosamler}, Schatte \cite{Schatte} independently in 1988.  The present form appearing in the above definition was stated by Lacey and Philipp \cite{LP1990} in 1990. And in 1999, Ibragimov and Lifshits \cite{ILpaper} gave the above sufficient condition.

Using this  criterion, the authors of \cite{BNT10} established the ASCLT for  functionals over general Gaussian fields. The Malliavin-Stein approach plays a crucial role in their work. Later, C. Zheng proved the ASCLT on the Poisson space in his Ph.D thesis  \cite[Chapter 5]{cengbo_thesis}.  And in this work, we prove the ASCLT in the Rademacher setting, see Section 3.2.

The rest of this paper is organised as follows: Section 2 is devoted to some preliminary knowledge on Rademacher functionals,  and we provide a simple but useful approximate chain rule there.  In Section 3.1, we establish the Wasserstein distance bound for normal approximation in  both setting (S) and (NS); in Section 3.2,  the ASCLT for Rademacher chaos is established.

 \section{Preliminaries}

We fix several notation first: $\N = \{1, 2,\cdot\cdot\cdot \}$,   $\textbf{Y}  = (Y_k, k\in\N )$ stands for the Rademacher sequence in the setting (S), and it also means the normalised sequence in the setting (NS).  Denote by $\G$ the $\sigma$-algebra generated by $\textbf{Y}$, for notational convenience, we write $L^2(\Omega)$ for $L^2(\Omega, \G, \P)$ in the sequel. We write $\H = \ell^2(\N)$ for the Hilbert space of square-summable sequences indexed by $\N$.  $\H^{\otimes n}$ means the $n^{th}$ tensor product space and $\H^{\odot n}$ its symmetric subspace. We denote $\H^{\odot n}_0 = \big\{ f\in\H^{\odot n} \,:\, f\vert_{\Laplace_n^c} = 0 \big\}$ with   $\Laplace_n = \big\{ (i_1, \cdot\cdot\cdot, i_n)\in\N^n \, :\, i_k \neq i_j$ for different $k,j \big\}$. Clearly, $\H^{\odot 0}_0 = \H^0 = \R$. For $u,v, \text{w}\in\H$, we write $\langle u, v\text{w} \rangle  = \sum_{k\in\N} u_k v_k \text{w}_k$.

\subsection{Discrete Malliavin calculus}

The basic reference for this section is the survey \cite{Privault1} by Privault.

\begin{defn}   The (discrete) $n^{th}$ order multiple stochastic integral $J_n(f)$ of $f\in\H^{\otimes n}$, $n\geq 1$, is given by  
 \begin{equation}\label{integral} J_n(f) = \sum_{(i_1, \cdot\cdot\cdot, i_n)\in\Laplace_n} f(i_1, \cdot\cdot\cdot, i_n) Y_{i_1}  Y_{i_2} \cdot\cdot\cdot Y_{i_n} \,\, . 
 \end{equation}
We define  $J_0(c)  = c$ for any $c\in\R$.  It is clear that $J_n(f) = J_n\big( \wt{f}\1_{\Laplace_n}\big)$, where $\wt{f}$ is  the standard symmetrisation of $f$.
\end{defn}
 For $g\in\H^{\odot n}_0$,  it is easy to check that $\| J_n(g) \| ^2_{L^2(\Omega)} = n!  \| g \| ^2_{\H^{\otimes n}}$ and $\CC_n : = \big\{ J_n(g)\, :\, g\in\H^{\odot n}_0 \big\}$ is isometric to $\big(\H^{\odot n}_0, \sqrt{n!} \| \cdot \| _{\H^{\otimes n}} \big)$. $\CC_n$ is called the \emph{Rademacher chaos} of order $n$, and one can see easily that $\CC_n$ is a closed linear subspace of $L^2(\Omega)$ and $\CC_n$, $\CC_m$ are mutually orthogonal for distinct $m,n$:
 \begin{align}\label{ISO}
 \E\big[ J_n(f) \cdot J_m(g) \big] = n! \cdot \big\langle f, g \big\rangle_{\H^{\otimes n}}\cdot \1_{(m=n)} , \qquad \forall f\in\H^{\odot n}_0 \,\,, \,\, g\in\H^{\odot m}_0 \,\, . 
 \end{align}

\paragraph{More notation} $\CC_0 := \R$.   We denote by $\mathcal{S}$ the linear subspace of $L^2(\Omega)$ spanned by multiple integrals  and it is a well-known result (\emph{e.g.} see \cite[Proposition 6.7]{Privault1}) that $\mathcal{S}$ is dense in $L^2(\Omega)$.  In particular,  $F\in L^2(\Omega)$ can be expressed as follows:
 \begin{align}\label{chaoticdecom}F = \E[F] + \sum_{n\geq 1} J_n(f_n) \,\,,\quad \text{where $f_n\in\H^{\odot n}_0$ for each $n\in\N$.} \end{align}
We denote by  $L^2(\Omega\times\N)$  the space of  square-integrable random sequences $a  = (a_k, k\in\N )$, where $a_k$ is a real random variable for each $k\in\N$ and $\| a \| _{L^2(\Omega\times\N)}^2 := \E\big[  \| a \| _\H^2 \big] =  \sum_{k\geq 1} \E\big[ a_k^2 \big]  < +\infty$.

\begin{defn}    $\DD$ is the set of random variables $F\in L^2(\Omega)$  as in  \eqref{chaoticdecom} satisfying $ \sum_{\ell = 1}^\infty \ell \ell! \| f_\ell \| ^2_{\H^{\otimes\ell}} < +\infty$.  For $F\in\DD$ as in \eqref{chaoticdecom}, $D_k F = \sum_{n\geq 1} n J_{n-1}(f_n(\cdot, k))$ for each $k\in\N$. $DF = (D_kF, k\in\N)$ is called the  discrete Malliavin derivative of $F$.
\end{defn}

\begin{rem}    Using \eqref{chaoticdecom} and \eqref{ISO}, we can obtain the \emph{Poincar\'e inequality} for $F\in\DD$, $\Var(F) \leq \E\big[ \| DF \| _\H ^2 \big]$.       \end{rem}

\begin{defn}  We define the {\it divergence operator} $\delta$ as the adjoint operator of $D$. We say  $u\in \text{dom}\delta \subset L^2(\Omega\times\N)$ if there exists some constant $C$ such that $\vert \E[ \langle u, DF \rangle_\H ] \vert \leq C \| F \| _{L^2(\Omega)}$ for any $F\in\DD$. Then it follows from the Riesz's representation theorem that there exists a unique element in $L^2(\Omega\times\N)$, which we denote by $\delta(u)$, such that the duality relation \eqref{duality} holds for any $F\in\DD$:
\begin{align} \label{duality} 
  \E\big[ \langle u, DF \rangle_\H \big] = \E\big[ F \delta(u) \big] \, . 
  \end{align}

\end{defn}

\begin{defn} We define the {\it Ornstein-Uhlenbeck operator} $L$ by $L = -\delta D$.   Its domain is given by $\text{dom}L =\{ F \in L^2(\Omega)$ admits the chaotic decomposition as in \eqref{chaoticdecom} such that  $\sum_{n=1}^\infty n^2 n! \cdot\| f_n \| ^2_{\H^{\otimes n}}$ is finite$\}$.   For centred $F\in L^2(\Omega)$  as in \eqref{chaoticdecom}, we define $L^{-1}F = -\sum_{n=1}^\infty  n^{-1} J_n(f_n)$.
 It is clear that for such a $F$, one has $LL^{-1}F = F$.   We call $L^{-1}$  the pseudo-inverse of $L$.

\end{defn}

Here is another look at the derivative operator $D$.

\begin{rem}  We choose $\Omega = \{ +1, -1 \}^\N$ and define $\P = \bigotimes_{k\in\N} \big( p_k \delta_{+1} + q_k \delta_{-1} \big)$. Then the coordinate projections $\omega = \{ \omega_1, \cdot\cdot\cdot \} \in\Omega \lmto  \omega_k = : X_k(\omega)$ is an independent sequence of non-symmetric, non-homogeneous Rademacher random variables under $\P$.  We can define for $F\in L^2(\Omega)$, $F^{\oplus k}: = F(\omega_1, \cdot\cdot\cdot, \omega_{k-1}, 1, \omega_{k+1}, \cdot\cdot\cdot)$, that is, by fixing the $k^{th}$ coordinate in the configuration $\omega$ to be $1$. Similarly, we define  $F^{\ominus k}: = F(\omega_1, \cdot\cdot\cdot, \omega_{k-1}, -1, \omega_{k+1}, \cdot\cdot\cdot)$.  It holds that $D_kF = \sqrt{p_kq_k}\big( F^{\oplus k} - F^{\ominus k} \big)$, see \emph{e.g.} \cite[Proposition 7.3]{Privault1}. The following results are also clear :
 \begin{itemize}
 \item[$\bullet$] $\bv F^{\oplus k} - F \bv = \1_{(X_k = - 1)} \cdot \vert D_kF\vert / \sqrt{p_kq_k} \leq    \vert D_kF\vert/  \sqrt{p_kq_k} $ and $\bv F^{\ominus k} - F \bv =  \1_{(X_k =  1)} \cdot  \vert D_kF\vert/ \sqrt{p_kq_k}    \leq    \vert D_kF\vert/  \sqrt{p_kq_k}$.

\item[$\bullet$]  $F\in\DD$ if and only if $\sum_{k\in\N} p_kq_k \E\big[ \vert F^{\oplus k} - F^{\ominus k}\vert^2 \big] < +\infty$. In particular,  if $f:\R\to\R$ is Lipschitz continuous, then  $f(F)\in\DD$. 

\end{itemize}

\end{rem}

The following integration-by-part formula is important for our work.

\begin{lem}(\cite[Lemma 2.12]{NPR})\label{IBP} For every centred $F,G\in\DD$ and $f\in C^1(\R)$ with $\| f' \| _\infty < +\infty$,  one has $f(F),L^{-1}F\in\DD$ and $\E\big[ Gf(F) \big] = \E\big[ \langle -DL^{-1}G, Df(F) \rangle_\H \big]$.      \end{lem}
In particular, for $f(x) = x$, $\E\big[F^2\big]  =  \E\big[ \langle -DL^{-1}F, DF \rangle_\H \big]$.  The random variable $\langle -DL^{-1}F, DF \rangle_\H$ is crucial in the Malliavin-Stein approach, see \emph{e.g.} \cite{NP1} and \cite{PSTU}.

The term $Df(F)$ is not equal to $f'(F)DF$ in general, unlike the chain rule on Gaussian Wiener space,  see \emph{e.g.} \cite[Proposition 2.3.7]{bluebook}.  The following is our new approximate chain rule.

\begin{lem}\label{chainrule} {\bf\small (Chain rule)}  If $F\in\DD$ and $f:\R\lto\R$ is Lipschitz and of class $C^1$ such that  $f'$ is Lipschitz continuous, then
 \begin{align}
 D_kf(F) = f'(F)D_k F + R_k , \label{remainder_term_NS}
\end{align}
where the remainder term $R_k$ is bounded by $ \dfrac{\| f'' \| _\infty}{2\sqrt{p_kq_k}}   \cdot \vert D_k F \vert^2$ in the setting (NS).
\end{lem}

\begin{proof} Note first  $f(F)\in\DD$, since $f$ is Lipschitz. Moreover, since $f$ is of class $C^1$ with Lipschitz derivative, it follows immediately  that 
 $f(y) - f(x) = f'(x)(y-x) + \mathcal{R}(f)$, where
 the remainder term $\mathcal{R}(f)$ is bounded by $\| f'' \| _\infty \cdot \vert y - x \vert^2/2$. Therefore, in the setting (NS)
\begin{align*}
 D_kf(F) &= \sqrt{p_kq_k} \cdot\big[ f( F^{\oplus k} ) -   f( F^{\ominus k} ) \big] \\
 &= \sqrt{p_kq_k} \cdot \Big\{ f( F^{\oplus k} ) - f(F) - \big[   f( F^{\ominus k} )  -f(F) \big] \Big\} \\
 &=\sqrt{p_kq_k} \cdot \Big\{ \,   f'(F)(  F^{\oplus k} -   F ) + \mathcal{R}_{1,k} -  f'(F)(  F^{\ominus k} -   F ) + \mathcal{R}_{2,k} \Big\}
 \end{align*}
 with $\vert \mathcal{R}_{1,k}\vert \leq  \dfrac{\| f'' \| _\infty\cdot\vert D_kF\vert^2}{2p_kq_k}\cdot \1_{(X_k = -1)}$ and $\vert \mathcal{R}_{2,k}\vert \leq  \dfrac{\| f'' \| _\infty\cdot\vert D_kF\vert^2}{2p_kq_k}\cdot \1_{(X_k = 1)}$.
 
 Whence, $ D_kf(F) = \sqrt{p_kq_k} \cdot f'(F) \big(   F^{\oplus k}  -   F^{\ominus k} \big) + R_k = f'(F)D_kF + R_k$ and the remainder term $R_k = \sqrt{p_kq_k}\big( \mathcal{R}_{1,k}  + \mathcal{R}_{2,k} \big)  $ is bounded by $\dfrac{\| f''\| _\infty\cdot\vert D_kF\vert^2}{2\sqrt{p_kq_k}}$.        \end{proof}

It is clear that in the setting  (S), the remainder $R_k$ in \eqref{remainder_term_NS}  is bounded by $\| f'' \| _\infty \cdot \vert D_k F \vert^2$.

\begin{rem} In the setting (S), our approximate chain rule is different from that developed in \cite{NPR}, in which $f$ is assumed to be of class $C^3$ such that $f(F)\in\DD$ and $\| f'''\| _\infty < +\infty$. Moreover,  their chain rule is given as follows:
  \begin{align}\label{oldchain}
  D_kf(F) = f'(F) D_kF - \frac{1}{2} \Big[ f''(F^{\oplus k}) +   f''(F^{\ominus k}) \Big]\cdot (D_kF)^2\cdot Y_k + \wt{R}_k 
  \end{align}
with $\bv \wt{R}_k  \bv \leq \frac{10}{3}\| f''' \| _\infty  \bv D_kF\bv^3$. Apparently, ours is neater and  requires less regularity of $f$. This is important when we try to get some nice distance bound in Section 3.1. Following \cite{NPR}, the authors of \cite{PT15} gave an approximate chain rule in the setting (NS): if $f$ is  of class $C^3$ such that $f(F)\in\DD$ and $\| f'''\| _\infty < +\infty$, then 
  \begin{align}\label{oldchain_B}
  D_kf(F) = f'(F) D_kF - \frac{\vert D_kF\vert^2}{4\sqrt{p_kq_k}} \Big[ f''(F^{\oplus k}) +   f''(F^{\ominus k}) \Big]\cdot (D_kF)^2\cdot X_k + R^F_k 
  \end{align}
with the remainder term $R^F_k $ bounded by $\dfrac{5}{3!} \| f''' \| _\infty \dfrac{\vert D_kF\vert^3}{p_kq_k}$.  See   also Remark \ref{old_new_chain_rule}.

\end{rem}

\subsection{Star-contractions}

 Fix $m,n\in\N$, and   $r = 0, \cdot\cdot\cdot, n\wedge m$.  For $f\in\H^{\otimes n}$ and $g\in\H^{\otimes m}$,    $f \star^r_r g$ is an element in $\H^{\otimes n+m - 2r}$  defined by  
\begin{align*}
   f \star^r_r g(i_1, \cdot\cdot\cdot, i_{n-r},  j_1, \cdot\cdot\cdot, j_{m-r} ) =  \sum_{a_1, \cdot\cdot\cdot, a_r\in\N} f\big(  i_1, \cdot\cdot\cdot, i_{n-r},  a_1, \cdot\cdot\cdot, a_r \big)   g\big(  j_1, \cdot\cdot\cdot, j_{m-r}, a_1, \cdot\cdot\cdot, a_r \big)  \,\, . 
             \end{align*}

\begin{lem}\label{easy_lemma}  Fix $\ell\in\N$ and $0\leq r \leq \ell$. If $f, g\in\H^{\odot \ell}$, then 
 $$2\big\| f \star^r_r g \big\| ^2_{\H^{\otimes 2\ell - 2r}} \leq   \big\| f \star^{\ell-r}_{\ell-r} f \big\| ^2_{\H^{\otimes 2r}} +   \big\| g \star^{\ell-r}_{\ell-r} g \big\| ^2_{\H^{\otimes 2r}} \, .$$
In particular, $\big\| f \star^r_r g \big\| _{\H^{\otimes 2\ell - 2r}} \leq    \big\| f \star^{\ell-r}_{\ell-r} f \big\| _{\H^{\otimes 2r}} +   \big\| g \star^{\ell-r}_{\ell-r} g \big\| _{\H^{\otimes 2r}}$.
\end{lem}

\begin{proof} It follows easily  from the definition that
\begin{align*}
2\big\| f \star^r_r g \big\| ^2_{\H^{\otimes 2\ell -2r}} &= 2 \big\langle f \star^{\ell-r}_{\ell-r} f , g \star^{\ell-r}_{\ell-r} g \big\rangle_{\H^{\otimes 2r}} \\
 &\leq  \big\| f \star^{\ell-r}_{\ell-r} f \big\| ^2_{\H^{\otimes 2r}} +   \big\| g \star^{\ell-r}_{\ell-r} g \big\| ^2_{\H^{\otimes 2r}}  .
\end{align*}
\end{proof}

\subsection{Stein's method of normal approximation}

A basic reference for Stein's method is the monograph \cite{CGS}.  Let us start with a fundamental fact that  a real integrable random variable   $Z$ is a standard Gaussian random variable if and only if  $\E\big[ f'(Z) \big] = \E\big[ Z \cdot f(Z) \big]$ for every  bounded differentiable  function $f : \R\lto \R$.

Now suppose that $Z\sim\mathscr{N}(0,1)$,  for $h:\R\lto\R$ measurable such that $\E\bv h(Z)\bv < +\infty$, the differential equation $f'(x) - xf(x) = h(x) - \E[ h(Z) ]$ with unknown $f$ is called  the \emph{Stein's equation} associated with $h$. We call $f$ its solution, if $f$ is absolutely continuous and one version of $f'$ satisfies the Stein's equation everywhere. More precisely, we take $f'(x) = xf(x) + h(x) - \E\big[ h(Z)\big]$ for every $x\in\R$.

 It  is well known (see \emph{e.g.} \cite[Chapter 2]{CGS}, \cite[Chapter 3]{bluebook}) that given such a function $h$, there exists a unique solution $f_h$ to the Stein's equation such that $\lim_{\vert x \vert \to+\infty} f_h(x) e^{-x^2/2} = 0$.  Given a suitable \emph{separating}  class $\F$ of nice functions, we  define 
$$d_\F\big( X, Z \big) := \sup_{f\in\F} \bv \E\big[ f(X) \big] -  \E\big[ f(Z) \big] \bv \,\, . $$
When $\F$ is set of $1$-Lipschitz functions, $d_\F$ is  the Wasserstein distance; when $\F$ is set of $1$-Lipschitz functions that are also uniformly bounded by $1$, $d_\F$ is called the Fortet-Mourier distance; when $\F$ is the collection of indicator functions $\1_{(-\infty, z]}$, $z\in\R$, $d_\F$ corresponds to the Kolmogorov distance appearing in the Berry-Ess\'een bound. We denote by $d_W, d_{\text{FM}}, d_{K}$ respectively  these  distances. It is trivial that $d_{FM}\leq d_W$, and it is not difficult to show that $d_K(X, Y) \leq \sqrt{2C\cdot d_W(X,Y)}$ if $X$ has density function uniformly bounded by $C$.

Now we replace the dummy variable $x$ in the Stein's equation by a generic random variable $X$, then taking expectation on both sides of the equation gives $\E[ Xf_h(X) - f_h'(X) ] = \E[ h(X) - h(Z)]$.   

Here we collect several bounds for the Stein's solution $f_h$:

\begin{itemize}

\item For $h:\R\lto\R$ $1$-Lipschitz, $f_h$ is of class $C^1$ and $f_h'$ is bounded Lipschitz with $\| f_h' \| _\infty \leq \sqrt{2/\pi}$, $\| f_h'' \| _\infty \leq 2$, see \emph{e.g.} \cite[Lemma 4.2]{Chatterjee08}.  We denote by $\F_W$ the family of differentiable functions $\phi$ satisfying $\| \phi' \| _\infty \leq \sqrt{2/\pi}$, $\| \phi'' \| _\infty \leq 2$, therefore  for any square-integrable random variable $F$,
\begin{align}\label{d_W}
d_{\text{FM}}(F,Z) \leq d_W(F, Z)   \leq  \sup_{\phi\in\F_W}  \Bv \E\big[ F\phi(F) - \phi'(F) \big] \Bv \,\, .
\end{align}

\item  If $h= \1_{(-\infty, z]}$ for some $z\in\R$, then $0 < f_h \leq \frac{\sqrt{2\pi}}{4}$ and $\| f_h' \| _\infty \leq 1$, see  \cite[Lemma 2.3]{CGS}. We  write $\F_K:=\big\{ \phi \,: \, \| \phi' \| _\infty \leq 1$, $\| \phi \| _\infty \leq \frac{\sqrt{2\pi}}{4} \big\}$, therefore  for any integrable random variable $F$,
\begin{align}\label{d_K}
 d_K(F, Z)   \leq  \sup_{\phi\in\F_K}  \Bv \E\big[ F\phi(F) - \phi'(F) \big] \Bv \,\, .
\end{align}
As the density  of $Z$ is uniformly bounded by $1/\sqrt{2\pi}$, we have the easy bound $ d_K(F, N) \leq \sqrt{d_W(F,N)}$.

\end{itemize}

\section{Main results}

\subsection{Normal approximation  in Wasserstein distance}


In this subsection, we derive the Wasserstein distance bound for normal approximation of Rademacher functionals. Our new chain rule plays a crucial role.

\begin{thm}\label{NPbdd} Given $Z\sim\mathscr{N}(0,1)$ and $F\in\DD$ centred,   one has in the setting (NS) that
\begin{align} \label{Napp}
  d_W(F,Z)  \leq \sqrt{\frac{2}{\pi}} \cdot \E\Big[ \bv 1 - \big\langle  DF, - DL^{-1} F \big\rangle_\H \bv\Big] +  \E\left[\,   \Big\langle\,  \frac{1}{\sqrt{pq}} \vert DL^{-1} F \vert , \vert DF\vert^2 \Big\rangle_\H  \,\, \right]\,.
\end{align}
  In particular, if $F\in\CC_m$ for some  $m\in\N$, then 
\begin{align} \label{Gausspart}
 \E\Big[ \bv 1 - \big\langle  DF, - DL^{-1} F \big\rangle_\H \bv\Big] \leq \bv 1 - \E[F^2] \bv + \frac{1}{m} \sqrt{\Var\big( \| DF \| _\H^2 \big) }
\end{align}
and
\begin{align}\label{infpart}
\E\left[\,   \Big\langle\,  \frac{1}{\sqrt{pq}} \vert DL^{-1} F \vert , \vert DF\vert^2 \Big\rangle_\H  \,\, \right] \leq \sqrt{\E[F^2]/m} \cdot \sqrt{ \sum_{k\in\N}  \frac{1}{p_kq_k}  \E\big[ \vert D_kF\vert^4  \big] } \, .
\end{align}
\end{thm}

\begin{proof}  Given $\phi\in\F_W$, it follows from \ref{IBP} and \ref{chainrule} that 
\begin{align*}
 \E\big[ F\phi(F)  \big]  & =  \E\big[  \langle D\phi(F), -DL^{-1}F \rangle_\H   \big]  =  \E\big[  \phi'(F)\langle DF, -DL^{-1}F \rangle_\H + \langle R, -DL^{-1}F\rangle_\H   \big] ,
 \end{align*}
where $R = (R_k, k\in\N)$ is the remainder satisfying $\vert R_k\vert \leq \vert D_kF\vert^2/\sqrt{p_kq_k}$. Thus,
\begin{align*}
  \E\big[ F\phi(F)  \big] - \E[ \phi'(F)] = \E\Big[ \phi'(F) \Big( \langle DF, -DL^{-1}F \rangle_\H - 1 \Big) \Big] +  \E\big[ \langle R, -DL^{-1}F\rangle_\H   \big]    
\end{align*}
implying that 
$$\Bv\E\big[ F\phi(F) - \phi'(F) \big]  \Bv \leq \sqrt{\frac{2}{\pi}} \cdot  \E\Big[ \bv 1 - \big\langle  DF, - DL^{-1} F \big\rangle_\H \bv\Big] +     \E\left[\,   \Big\langle\,  \frac{1}{\sqrt{pq}} \vert DL^{-1} F \vert , \vert DF\vert^2 \Big\rangle_\H  \,\, \right] \,\, . $$
Hence \eqref{Napp} follows from \eqref{d_W}.

 If $F = J_m(f)$ with $m\in\N$, $f\in\H^{\odot m}_0$, then $D_kF = m J_{m-1}\big[f(\cdot, k)\big]$, $D_kL^{-1}F = - J_{m-1}\big[f(\cdot, k)\big]$. Recall $\E[F^2]  =  \E[ \langle DF, -DL^{-1}F \rangle_\H]$, thus \eqref{Gausspart} follows easily from triangle inequality and Cauchy-Schwarz inequality:
 \begin{align}
  \E\Big[ \bv 1 - \big\langle  DF, - DL^{-1} F \big\rangle_\H \bv\Big] & \leq \bv 1 - \E[F^2] \bv +  \E\Big[ \bv \E[F^2] -  \langle DF, -DL^{-1}F \rangle_\H \bv  \Big]  \notag \\
  &\leq \bv 1 - \E[F^2] \bv +   \sqrt{\Var\big( \langle DF, -DL^{-1}F \rangle_\H   \big) } \notag \\
  & = \bv 1 - \E[F^2] \bv + \frac{1}{m} \sqrt{\Var\big( \| DF \| _\H^2 \big) } \, . \notag
 \end{align}
 The inequality \eqref{infpart} is also an easy consequence of Cauchy-Schwarz inequality:
 \begin{align*} 
  \E\left[\,   \Big\langle\,  \frac{1}{\sqrt{pq}} \vert DL^{-1} F \vert , \vert DF\vert^2 \Big\rangle_\H  \,\, \right]  & =  \frac{1}{m} \sum_{k\in\N}   \E\left[ \,  \vert D_k F \vert \cdot   \frac{1}{\sqrt{p_kq_k}}  \vert D_kF\vert^2 \,\right] \\
  & \leq  \frac{1}{m}\sqrt{ \sum_{k\in\N} \E\big[ \vert D_k F \vert^2 \big]} \cdot \sqrt{ \sum_{k\in\N}  \frac{1}{p_kq_k}  \E\big[ \vert D_kF\vert^4  \big] }  \\
  & = \sqrt{\frac{\E[F^2]}{m}} \cdot \sqrt{ \sum_{k\in\N}  \frac{1}{p_kq_k}  \E\big[ \vert D_kF\vert^4  \big] } \,\, .
   \end{align*}

\end{proof}

\begin{rem}\label{old_new_chain_rule} (1) In the setting (S), the results in \ref{NPbdd} can be easily deduced by taking $p_k = q_k =1/2$ for each $k\in\N$. As we have mentioned earlier, our approximate chain rule  is neater than \eqref{oldchain} given in \cite[Proposition 2.14]{NPR}, since it requires less regularity (this is the key point for us to get the estimate in Wasserstein distance).  Although  the authors of \cite{NPR} were able to derive the Wasserstein distance via some smoothing argument,  they imposed  some further assumption and their rate of convergence is suboptimal compared to ours.

(2)    In the setting (NS), Privault and Torrisi used their approximate chain rule \eqref{oldchain_B} and the smoothing argument to obtain the Fortet-Mourier distance, see \cite[Section 3.3]{PT15}. It is suboptimal compared to our estimate in Wasserstein distance, in view of the trivial relation $d_{\text{FM}}\leq d_W$. 

(3)  Recall that the test function $\phi\in\F_K$ (see \eqref{d_K}) may not have Lipschitz derivative, so our approximate chain rule as well as those in \cite{NPR,PT15} does not work to achieve the bound in Kolmogorov distance. Instead of using the chain rule, the authors of \cite{KRC14}  carefully used a representation of the discrete Malliavin derivative $D\phi(F)$ and the fundamental theorem of calculus, this turns out to be flexible enough for them  to deduce the Berry-Ess\'een bound in the setting (S). Later they obtained the Berry-Ess\'een bound in the setting (NS) with applications  to   random graphs. One can easily see that  two terms in \eqref{BSbdd1} are almost the same as our bound in Wasserstein distance while there are two  extra terms \eqref{BSbdd2} in their Kolmogorov distance bound.

\end{rem}

Due to a comparison between \eqref{Kkkk_bdd} and \eqref{Napp},  we are able to replace the Kolmogorov distance in many statements in \cite{KRC14, KRC15} by the Wasserstein distacne (with fewer terms and slightly different multiplicative constants). For example, we will obtain the so-called second-order Poincar\'e inequality in Wasserstein distance in the following

\begin{rem}({\bf\small Second-order Poincar\'e inequality}) \quad  One can apply the Poincar\'e's inequality  to  $\langle -DL^{-1} F ,  DF \rangle_\H$:
 \begin{align}\label{hehe}
 \Var\big(  \langle -DL^{-1} F ,  DF \rangle_\H \big) \leq \E\Big[  \big\| D  \langle -DL^{-1} F ,  DF \rangle_\H  \big\| _\H^2 \Big],
 \end{align}
provided $ \langle -DL^{-1} F ,  DF \rangle_\H \in\DD$.

In \cite{KRC15}, Krokowski \emph{et al.} gave the bound on Kolmogorov distance, see \eqref{Kkkk_bdd}; they also established the so-called second-order Poincar\'e inequality as follows.  For $Z\sim\mathscr{N}(0,1)$ and $F\in\DD$ centred with unit variance, and $r,s,t\in(1,\infty)$ such that $r^{-1} +  s^{-1} + t^{-1} = 1$, it holds  that
\begin{align}\label{Second_K}
d_K\big(F,  Z\big) \leq A_1 + A_2 + \frac{\sqrt{2\pi}}{8} \cdot A_3 + A_4 + A_5 + A_6 + A_7 \, \, ,
\end{align}
where 
\begin{align*}
A_1 : &= \left( \,\, \frac{15}{4} \sum_{j,k,\ell =1}^\infty \sqrt{ \E\big[  (D_jF)^2(D_kF)^2   \big]  }  \sqrt{ \E\big[  (D_\ell D_jF)^2(D_\ell D_kF)^2   \big]  }  \,\, \right)^{1/2} ; \\
A_2 : & = \left( \,\,  \frac{3}{4}  \sum_{j,k,\ell=1}^\infty \frac{1}{p_\ell q_\ell}  \E\Big[  \, (D_\ell D_jF)^2(D_\ell D_kF)^2  \, \Big] \,\, \right)^{1/2}  ; \\
A_3 : &= \sum_{k=1}^\infty \frac{1}{\sqrt{p_kq_k}}  \E\big[ \vert D_k F \vert^3 \big] \,\,  ;  \quad A_4: = \frac{1}{2}  \big\| F \big\| _{L^r(\Omega)}  \sum_{k=1}^\infty  \frac{1}{\sqrt{p_kq_k}}  \big\| (D_kF)^2 \big\| _{L^s(\Omega)}  \big\| D_kF \big\| _{L^t(\Omega)} ; \\
A_5 : & = \left( \,\, \sum_{k\in\N} \frac{1}{p_kq_k} \E\Big[ (D_kF)^4 \Big] \, \right)^{1/2} \,\, ; \quad A_6 = \left( \,\, 3  \sum_{k,\ell=1}^\infty \frac{1}{p_kq_kp_\ell q_\ell}  \E\Big[ \big( D_\ell D_kF \big)^4 \,\Big] \,\, \right)^{1/2} ;\\
A_7 : & = \left( \,\, 6  \sum_{k,\ell=1}^\infty \frac{1}{p_kq_k}  \sqrt{ \E\big[ (D_kF)^4\big]  } \sqrt{\E\big[ (D_\ell D_kF)^4 \big]} \,\, \right)^{1/2} \,\, .
\end{align*}
Part of the proof for \eqref{Second_K} requires fine analysis of the discrete gradients in \eqref{hehe}. For more details, see \cite[Theorem 4.1]{KRC15}.  Following exactly the same lines, one can obtain the following second-order Poincar\'e inequality in Wasserstein distance by simply comparing \eqref{Kkkk_bdd} and \eqref{Napp}, as well as going through the proof of \cite[Theorem 4.1]{KRC15}:  $d_W(F, Z) \leq \sqrt{2/\pi} \cdot ( A_1 + A_2) + A_3$. Note the constants $r, s, t$ are not involved in $A_1, A_2, A_3$.

\end{rem}

In the end of this subsection, we recall from \cite{NPR}    sufficient conditions for CLT (in the  setting (S)) inside a fixed Rademacher chaos $\CC_\ell$ ($\ell\geq 2$). The analogous result in  the setting (NS) was proved in \cite[Section 5.3]{PT15}.

\begin{prop} (\cite[Theorem 4.1]{NPR})  In the setting (S), fix $\ell\geq 2$. If $F_n : = J_\ell(f_n)$ for some  $f_n\in\H^{\odot \ell}_0$, then 
\begin{align}
\Var\big( \| DF_n \| _\H^2 \big) \leq  C \sum_{m=1}^{\ell-1} \big\| f_n\star_m^m f_n \big\| ^2_{\H^{\otimes 2\ell-2m}} \, , \label{2221}
\end{align}
and 
\begin{align}
 \sum_{k\in\N} \E\big[ \bv D_kF_n\bv^4\big]  \leq  C \sum_{m=1}^{\ell-1}   \big\| f_n\star_m^m f_n \big\| ^2_{\H^{\otimes 2\ell-2m}} ,  \quad \label{4443} 
\end{align}
where the   constant $C$ only depends on $\ell$.  
\end{prop}
As a consequence, the following result is straightforward.  

\begin{prop}\label{nishabi}(\cite[Proposition 4.3]{NPR}) If $ \| f_n \| ^2_{\H^{\otimes \ell}}\cdot \ell! \lto 1$ and 
 \begin{align}\label{CLT_suff} 
  \big\| f_n\star^m_m f_n \big\| _{\H^{\otimes (2\ell-2m)}} \lto 0 \,\, ,\quad\forall m = 1, 2, \cdot\cdot\cdot, \ell-1 \,\, ,
 \end{align}
as $n\to+\infty$, then $F_n$ converges in law to a standard Gaussian random variable.  \end{prop}

\subsection{Almost sure central limit theorem for Rademacher chaos}

The following  lemma  is crucial for us to apply the Ibragimov-Lifshits criterion.  The Gaussian analogue was proved in \cite[Lemma 2.2]{BNT10} and the Poisson case was given in \cite[Proposition 5.2.5]{cengbo_thesis}.

\begin{lem}\label{ILbdd} In  the setting (NS),  if $F\in\DD$ is centred such that $  \langle -DL^{-1} F ,  DF \rangle_\H \in L^2(\Omega)$ and
       $$ \sum_{k\in\N}\frac{1}{p_kq_k} \E\big[\vert D_kF\vert^4\big]  < +\infty\,\, ,$$ 
then
\begin{align}
  \big\vert \E\big[ e^{itF}\big] -  e^{-t^2/2} \big\vert & \leq    \vert t \vert ^2 \cdot  \bv 1 - \E[F^2] \bv  +  t^2 \cdot \sqrt{\Var\big(    \langle -DL^{-1} F ,  DF \rangle_\H     \big)} \notag\\ 
   &\qquad +  \vert t\vert^3 \cdot \sum_{k\in\N}\frac{1}{\sqrt{p_kq_k}}\E\Big[ \vert D_k L^{-1} F \vert \cdot (D_kF)^2 \Big]  \,\, . \label{EST1}
\end{align}
In particular, if $F = J_\ell(f)$ for some  $\ell\in\N$ and $f\in\H^{\odot \ell}_0$, then
\begin{align} 
 \big\vert \E\big[ e^{itF}\big] -  e^{-t^2/2} \big\vert & \leq   \vert t \vert^2 \cdot \bv 1 - \ell! \| f \| _{\H^{\otimes \ell}}^2 \bv + \frac{\vert t\vert^2}{\ell} \sqrt{\Var\big( \| DF \| _\H^2 \big) } \notag \\
 &\qquad + \vert t\vert^3\cdot \sqrt{ \sum_{k\in\N}\frac{1}{p_kq_k} \E\big[\vert D_kF\vert^4\big] }  \sqrt{\E[ F^2]/\ell}  \,\,.\label{EST2}
\end{align}

\end{lem}

\begin{proof}
Set $\phi(t) = e^{t^2/2} \cdot \E\big[ e^{it F} \big]$, $t\in\R$. Then 
 \begin{align}\label{useful2} \big\vert \E\big[ e^{itF}\big] -  e^{-t^2/2} \big\vert = \vert \phi(t) - \phi(0) \vert \cdot e^{-t^2/2} \leq     e^{-t^2/2}  \cdot  \vert t \vert \cdot \sup_{ \vert s \vert \leq \vert t\vert } \bv \phi'(s) \bv \,\, .
 \end{align}
Clearly,  $$\phi'(t) =  t e^{t^2/2} \E[ e^{itF} ] + ie^{t^2/2} \E[ F \cdot e^{it F} ] = t e^{t^2/2} \E[ e^{itF} ] + ie^{t^2/2} \E[ \langle -DL^{-1} F , D e^{it F} \rangle_\H ] \, ,$$ and it follows from  \ref{chainrule} that for each $k\in\N$, 
 $$D_k e^{it F} = it e^{itF} D_k F  + R_k$$
  with $\vert R_k\vert \leq  \vert t \vert^2 \cdot \vert D_kF\vert^2/\sqrt{p_kq_k}$.  Therefore, 
\begin{align*}
 \phi'(t)  =  t e^{t^2/2} \E\big[ e^{itF} \big]   -t e^{t^2/2} \E\Big[ e^{itF}\cdot  \langle -DL^{-1} F ,  DF \rangle_\H \Big]  +  ie^{t^2/2} \E\big[ \langle -DL^{-1} F , R \rangle_\H \big] .  
\end{align*}
Then by triangle inequality,  one has 
$$\bv \phi'(t) \bv  \leq  \vert t\vert e^{t^2/2}\cdot \E\big[ \vert 1 -  \langle -DL^{-1} F ,  DF \rangle_\H \vert \big] + e^{t^2/2} \bv  \E\big[ \langle -DL^{-1} F , R \rangle_\H \big] \bv     . $$ 
 It follows from \eqref{useful2}  that
\begin{align*}
 \big\vert \E\big[ e^{itF}\big] -  e^{-t^2/2} \big\vert  \leq t^2\E\big[ \vert 1 -  \langle -DL^{-1} F ,  DF \rangle_\H \vert \big] + \vert t\vert^3\cdot \sum_{k\in\N}\frac{1}{\sqrt{p_kq_k}}\E\Big[ \vert D_k L^{-1} F \vert \cdot (D_kF)^2 \Big],
\end{align*}
then the desired inequality \eqref{EST1} follows from the estimate $\E\big[ \vert 1 -  \langle -DL^{-1} F ,  DF \rangle_\H \vert \big] \leq   \bv 1 - \E[F^2] \bv  +  \sqrt{\Var\big(    \langle -DL^{-1} F ,  DF \rangle_\H     \big)}$, see \ref{NPbdd}. The rest is straightforward.   \end{proof}

The following theorem provides sufficient conditions for the ASCLT on a fixed Rademacher chaos in the setting (S). The analogous results in the Gaussian and Poisson settings  can be found in \cite{BNT10, cengbo_thesis} respectively.

\begin{thm}\label{thm4.1} In the setting (S), fix $\ell \geq 2$ and $F_n =  J_\ell(f_n)$ with $f_n\in\H^{\odot \ell}_0$ for each $n\in\N$. Assume that $ \| f_n \| ^2_{\H^{\otimes \ell}}\cdot \ell! = 1$,  and the following two conditions  as well as \eqref{CLT_suff} are satisfied:
\begin{align*} 
  \text{\bf\small C-1}&\qquad \sum_{n\geq 2} \frac{1}{n\gamma_n^3} \sum_{k,j=1}^n \frac{\bv \langle f_k, f_j \rangle_{\H^{\otimes \ell}}\bv  }{kj} <+\infty \,\, ;\\
  \text{\bf\small C-2}&\qquad \sum_{n\geq 2} \frac{1}{n\gamma_n^2} \sum_{k=1}^n \frac{1}{k} \cdot \big\| f_k\star^m_m f_k \big\| _{\H^{\otimes 2\ell-2m}}    < +\infty \,\, ,\quad\forall m=1, \ldots, \ell-1 \, .
  \end{align*}
Then the ASCLT holds for $(F_n)_{n\in\N}$. 

\end{thm}

\begin{proof} Note first $F_n$ converges in law to a standard Gaussian random variable, by Proposition \ref{nishabi}.   Observe that 
\begin{align}
\bv \Laplace_n(t) \bv^2 & = \frac{1}{\gamma_n^2} \sum_{k,j=1}^n \frac{1}{kj} \Big( e^{it ( F_k - F_j )} - e^{-t^2} \Big) - \frac{e^{-t^2/2}}{\gamma_n} \sum_{k=1}^n \frac{1}{k} \Big( e^{it  F_k } - e^{-t^2/2} \Big) \label{For1} \\
& \qquad\qquad\qquad\qquad\qquad  -\frac{e^{-t^2/2}}{\gamma_n}  \sum_{j=1}^n \frac{1}{j} \Big( e^{-it F_j } - e^{-t^2/2} \Big) \,. \notag
\end{align}
Now we fix $r>0$, $t\in[-r, r]$.  For brevity, we omit the subscripts. One can deduce  from  \ref{ILbdd}   and \eqref{2221}, \eqref{4443} that 
\begin{align*}
  \big\vert \E\big[ e^{itF_k}\big] -  e^{-t^2/2} \big\vert  & \leq  C \cdot  \sqrt{   \sum_{m=1}^{\ell-1} \big\| f_k \star^m_m f_k \big\| ^2  } ,
  \end{align*}
here and in the following the constant $C$ may vary from line to line but only depend on $r, \ell$.

Since $\sqrt{a_1 + \ldots + a_l} \leq \sqrt{a_1} + \ldots + \sqrt{a_l}$ for any $ a_1, \ldots, a_l\geq 0$,  
$$ \big\vert \E\big[ e^{itF_k}\big] -  e^{-t^2/2} \big\vert    \leq C \cdot  \sum_{m=1}^{\ell-1} \big\| f_k \star^m_m f_k \big\| .$$   
  Similarly, we apply the same argument with $s = \sqrt{2} t$ and $g = (f_k - f_j)/\sqrt{2}$, and we  get 
\begin{align} \label{For2}
  \Bv \E\Big[ e^{it(F_k-F_j)}\Big] -  e^{-t^2} \Bv  =   \Bv \E\big[ e^{is \cdot J_\ell(g)  }\big] -  e^{-s^2/2} \Bv    \leq   C  \sum_{m=1}^{\ell-1}  \big\|  g \star^m_m g \big\|     +  C \cdot \bv  \langle f_k, f_j \rangle \bv   .\end{align}
Clearly,     $g \star^m_m g = \frac{1}{2}\big(  f_k \star^m_m f_k+ f_j \star^m_m f_j - f_k \star^m_m f_j - f_j \star^m_m f_k \big)$, then 
$$2 \big\| g \star^m_m g \big\|  \leq  \big\| f_k \star^m_m f_k \big\|  +  \big\| f_j \star^m_m f_j \big\|   +  2  \big\| f_k \star^m_m f_j \big\|  \,\, .    $$
 \ref{easy_lemma} implies   $   \big\| f_k \star^m_m f_j \big\|       \leq  \big\|   f_k\star^{\ell-m}_{\ell-m} f_k \big\|  +  \big\|   f_j\star^{\ell-m}_{\ell-m} f_j \big\| $. Therefore, 
\begin{align} 
  \Bv \E\Big[ e^{it(F_k-F_j)}\Big] -  e^{-t^2} \Bv &\leq  C   \bv \langle f_k, f_j \rangle \bv   + C     \sum_{m=1}^{\ell-1}   \big(  \big\| f_k \star^m_m f_k \big\|   +  \big\| f_j \star^m_m f_j \big\|          \big)  \, . \label{For3}
  \end{align}
Hence,  
\begin{align}
 &\qquad \E\big( \vert \Laplace_n(t) \vert^2 \big) \label{For4}\\
  &\leq \frac{C}{\gamma_n^2} \sum_{k,j=1}^n \frac{ \bv \langle f_k, f_j \rangle \bv }{kj}   +  \frac{C}{\gamma_n^2} \sum_{k,j=1}^n \frac{1}{kj}      \sum_{m=1}^{\ell-1} \big(    \big\| f_k \star^m_m f_k \big\|   +  \big\| f_j \star^m_m f_j \big\|    \big)     + \frac{C}{\gamma_n}   \sum_{k=1}^n \frac{1}{k}\cdot    \sum_{m=1}^{\ell-1}   \big\| f_k \star^m_m f_k \big\|    \notag \\
  & =  \frac{C}{\gamma_n^2} \sum_{k,j=1}^n \frac{ \bv \langle f_k, f_j \rangle \bv }{kj}         + \frac{2C}{\gamma_n}   \sum_{k=1}^n \frac{1}{k}\cdot    \sum_{m=1}^{\ell-1}   \big\| f_k \star^m_m f_k \big\|  \,\, . \notag
\end{align} 
Now we can see that  the conditions {\bf\small C-1, C-2} imply  the Ibragimov-Lifshits condition \eqref{ILcondition}, so the ASCLT holds for $(F_n, n\in\N)$.     \end{proof}

In the setting (S),  the normalised  partial sum $S_n = (Y_1 + \cdot\cdot\cdot + Y_n )/\sqrt{n}$ converges in law to a standard Gaussian random variable.  Moreover, the ASCLT holds for $(S_n)$, this is a particular case of \cite[Theorem 2]{LP1990}.  The following result is a slight generalisation of this classic example. 

\begin{cor}\label{ASCLT_1d}
In the setting (NS), let $F_n = J_1(f_n)$ be such that $\| f_n \| _\H = 1$ for all $ n\in\N$.  Assume that the following conditions hold:
\begin{align*} 
  \text{(i)}&\qquad \sum_{n\geq 2} \frac{1}{n\gamma_n^3} \sum_{k,j=1}^n \frac{\bv \langle f_k, f_j \rangle_\H \bv }{kj} < +\infty   \\
    \text{(ii)}&\qquad \sum_{n\geq 2} \frac{1}{n\gamma_n^2} \sum_{k=1}^n \frac{1}{k} \cdot   \sum_{m=1}^\infty \frac{1}{\sqrt{p_mq_m}} \vert f_k(m)\vert^3    < +\infty \\
      \text{(iii)}&\qquad  \sum_{m=1}^\infty \frac{1}{\sqrt{p_mq_m}} \vert f_k(m)\vert^3 \xrightarrow{k\to+\infty } 0 \,\, .
  \end{align*}
  Then  $F_n\xrightarrow{\text{law}} \mathscr{N}(0,1)$ and the ASCLT holds for $(F_n, n\in\N)$.
\end{cor}
 \begin{proof}  Note $-DL^{-1}F_n = f_n$ and $DF = f_n$. Then the quantity $ \langle -DL^{-1} F_n ,  DF_n \rangle_\H $ is deterministic and $\vert D_m L^{-1} F_n \vert\cdot (D_m F_n)^2 = \vert f_n(m)\vert^3$ for each $m\in\N$. Therefore, it follows from \ref{ILbdd} that
 $$ \big\vert \E\big[ e^{itF_n}\big] -  e^{-t^2/2} \big\vert \leq \vert t \vert^3 \sum_{m=1}^\infty \frac{1}{\sqrt{p_m q_m }}\vert f_n(m)\vert^3 \, .  $$
By (iii),  the CLT holds for $(F_n)$.  Similarly as in the proof of \ref{thm4.1}, we have 
\begin{align*}   \Bv \E\Big[ e^{it(F_k-F_j)}\Big] -  e^{-t^2} \Bv & \leq 2t^2 \bv \langle f_k, f_j \rangle_\H \bv + 2\sqrt{2} \vert t\vert^3  \sum_{m=1}^\infty \frac{1}{\sqrt{p_m q_m }}\vert f_k(m) - f_j(m)\vert^3  \\
& \leq 2t^2 \bv \langle f_k, f_j \rangle_\H \bv + 8\sqrt{2} \vert t\vert^3  \sum_{m=1}^\infty \frac{1}{\sqrt{p_m q_m }} \big( \vert f_k(m) \vert^3 + \vert f_j(m)\vert^3 \big) \,\, ,
\end{align*}
where the last inequality follows from the elementary inequality $(a + b)^3 \leq 4 a^3 + 4 b^3$ for any $a, b \geq 0$. The rest of the proof goes along the same lines as in \ref{thm4.1}.
\end{proof}

To conclude this section,  we give the following example as an application of \ref{thm4.1}.

\begin{example} In the setting (S), we consider the symmetric kernels $f_n\in \H^{\odot 2}_0$ for $n\geq 1$:
 \begin{align*}
  f_n(i,j)  = \begin{cases} 
\dfrac{1}{2\sqrt{n}} \quad\text{if $i,j\in\{1, 2,\cdot\cdot\cdot, 2n\}$ and $\vert i - j\vert = n$ ;}\\
 0 \qquad\text{otherwise.} 
  \end{cases}
 \end{align*}
Setting $F_n = J_2(f_n)$, we claim that the ASCLT holds for $(F_n, n\geq 1)$.  \bigskip
\begin{proof}
It is easy to get $2\| f_n \| _{\H^{\otimes 2} }^2 = 1$ and 
\begin{align*}
f_n\star^1_1 f_n(i,j) = \begin{cases}
\dfrac{1}{4n}   \quad \text{if $i=j\in\{1, 2,\cdot\cdot\cdot, 2n\}$;} \\
0 \qquad \text{otherwise.}
\end{cases}
\end{align*}
So $\| f_n\star^1_1 f_n \| _{\H^{\otimes 2}} = \dfrac{1}{2\sqrt{2n}}$ converges to zero as $n\to+\infty$, thus the CLT follows from Proposition \ref{nishabi}. If $k < \ell$,  then 
 $$\langle f_k, f_\ell \rangle_{\H^{\otimes 2}} = \sum_{i,j=1}^{2k} f_k(i,j) f_\ell(i,j) =  \sum_{i,j=1}^{2k} f_k(i,j) \1_{(\vert i - j \vert = k)} f_\ell(i,j) \1_{(\vert i-j\vert = \ell)} = 0 \,\, , $$
thus 
$$\sum_{n\geq 2} \frac{1}{n\gamma_n^3} \sum_{k,\ell=1}^n \frac{\bv \langle f_k, f_\ell \rangle_{\H^{\otimes 2}}\bv  }{k\ell} = \sum_{n\geq 2} \frac{1}{n\gamma_n^3} \sum_{k=1}^n \frac{\bv \langle f_k, f_k \rangle_{\H^{\otimes 2}}\bv  }{k^2}  = \sum_{n\geq 2} \frac{1}{n\gamma_n^3} \sum_{k=1}^n \frac{1 }{2k^2} < +\infty \,\, . $$
That is, the condition ({\bf C-1}) in \ref{thm4.1} is satisfied. It remains to check the condition ({\bf C-2}):
\begin{align*}
\sum_{n\geq 2} \frac{1}{n\gamma_n^2} \sum_{k=1}^n \frac{1}{k} \cdot \big\| f_k\star^1_1 f_k \big\| _{\H^{\otimes 2}}  =   \sum_{n\geq 2} \frac{1}{n\gamma_n^2} \sum_{k=1}^n \frac{1}{k} \cdot \dfrac{1}{2\sqrt{2k}}  <+\infty \,\, .
\end{align*}
Hence it follows from \ref{thm4.1} that the ASCLT holds for $(F_n, n\geq 1)$. \qedhere

\end{proof}

\end{example}

\end{document}